\documentclass[a4paper]{article}
\usepackage{amsmath, amsthm, amsfonts, amssymb, mathrsfs}
\usepackage[all]{xy}
\usepackage[pagebackref]{hyperref}
\usepackage{color}
\usepackage{authblk}

\theoremstyle{plain}
\newtheorem{theorem}{Theorem}[section]
\newtheorem{proposition}[theorem]{Proposition}
\newtheorem{corollary}[theorem]{Corollary}
\newtheorem{lemma}[theorem]{Lemma}
\theoremstyle{definition}
\newtheorem{definition}[theorem]{Definition}
\newtheorem{example}[theorem]{Example}
\newtheorem{remark}[theorem]{Remark}

\theoremstyle{plain}
\newtheorem*{theorem*}{Theorem}
\newtheorem*{proposition*}{Proposition}
\newtheorem*{corollary*}{Corollary}
\newtheorem*{lemma*}{Lemma}
\theoremstyle{definition}
\newtheorem*{definition*}{Definition}
\newtheorem*{remark*}{Remark}
\newtheorem*{notation*}{Notation}

\newcommand{\CC}{\mathbb{C}}
\newcommand{\PP}{\mathbb{P}}
\newcommand{\ZZ}{\mathbb{Z}}

\newcommand{\OO}{\mathcal{O}}
\newcommand{\UU}{\mathcal{U}}
\newcommand{\FF}{\mathcal{F}}
\newcommand{\EE}{\mathcal{E}}
\newcommand{\LL}{\mathcal{L}}

\newcommand{\KK}{\mathcal{K}}

\newcommand{\II}{\mathcal{I}}
\newcommand{\JJ}{\mathcal{J}}
\newcommand{\pp}{\mathfrak{p}}

\DeclareMathAlphabet{\mathpzc}{OT1}{pzc}{m}{it}

\newcommand{\codim}{\mathrm{codim}}
\newcommand{\ann}{\mathrm{ann}}
\newcommand{\proj}{\mathrm{Proj}}
\newcommand{\ext}{\mathrm{Ext}}
\renewcommand{\hom}{\mathrm{Hom}}
\newcommand{\sing}{\mathrm{Sing}(\omega)}
\newcommand{\per}{\mathpzc{Per}(\omega)}
\newcommand{\kup}{\mathpzc{Kup}(\omega)}

\title{Foliations with persistent singularities}
\author[1]{C\'esar Massri\thanks{The author was fully supported by CONICET, Argentina.}}
\author[2]{Ariel Molinuevo\thanks{The author was fully supported by Universidade Federal do Rio de Janeiro, Brasil.}}
\author[1]{Federico Quallbrunn$^*$}
\affil[1]{\small Departamento de Matem\'atica, Universidad CAECE, Argentina}
\affil[2]{\small Instituto de Matem\'atica, Universidade Federal do Rio de Janeiro, Brazil}
\date{}

\begin{document}

\maketitle

\renewcommand{\thefootnote}{\fnsymbol{footnote}} 
\renewcommand{\thefootnote}{\arabic{footnote}}

\begin{abstract}
Let $\omega$ be a differential $q$-form defining a foliation of codimension $q$ in a projective variety.
In this article we study the singular locus of $\omega$ in various settings. We relate a certain type of singularities, which we name \emph{persistent}, with the unfoldings of $\omega$, generalizing previous work done on foliations of codimension $1$ in projective space. We also relate the absence of persistent singularities with the existence of a connection in the sheaf of $1$-forms defining the foliation. 

\end{abstract}

\section*{Introduction}

\subsection*{Motivation and overview of the problem}

Foliations of arbitrary codimension over algebraic varieties have been considered for instance in the works of Malgrange \cite{malgrange, malgrange2} in the local case, and Jouanolou \cite{jou} in a more global approach. Aside from the main result of \cite{malgrange2} and general definitions, most of the early theorems about foliations on projective algebraic varieties have been formulated for codimension $1$ foliations on the projective space $\PP^n$. 
In those articles, codimension $q$ foliations were defined locally by $1$-forms $\omega_1,\dots,\omega_q$ satisfying Frobenius integrability equations: $d\omega_i\wedge\omega_1\wedge\dots\wedge\omega_q=0$ for $i=1,\dots,q$. This definition is not general enough for singular foliations of codimension $q$, as singular foliations by curves in dimension $n\geq 3$ cannot be given by $n-1$ forms even locally, see \cite{deMed1} and Example \ref{exampleCodimension2inA3}. The correct definition is given by a $q$-form verifying the Pl\"ucker relations and Frobenius integrability (see below for definitions).

As for why many results were stated with $\PP^n$ as ambient variety, notice that working in $\PP^n$ allows the use of homogeneous coordinates and so one can define a codimension $1$ foliation with an integrable polynomial $1$-form $\omega=\sum_i f_i(x) dx_i$, that is a $1$-form verifying $\omega\wedge d\omega=0$ and $\sum_i x_i f_i(x)=0$.  Such a setting can give concrete examples of foliations which may be hard to produce and study in more general contexts.

An important problem with many results in the codimension $1$ case and over $\PP^n$ but not so much in arbitrary codimension and over an arbitrary variety is the local and global characterization of the singularities of a foliation. Local results include the main theorems of \cite{malgrange} in codimension $1$ and of \cite{saito} and \cite{malgrange2} in higher codimension. Global studies have been made in the case of logarithmic foliations in \cite{fmi} and in the case of foliations defined by polynomial representations of affine lie algebras in \cite{ocgln} among others.
An important type of singularity of a holomorphic foliation was discovered by Ivan Kupka in \cite{kupka}. A Kupka singularity for an integrable $1$-form $\omega$ is a point $p$ such that $\omega(p)=0$ and $d\omega(p)\neq 0$. Kupka showed that this type of singularity of codimension $1$ foliation is stable, meaning that if $\omega_t$ is a family of integrable $1$-forms parameterized by $t$ and $\omega_0$ has a Kupka singularity then $\omega_t$ also has a Kupka singularity for small enough $t$.
Also if a foliation have a Kupka singularity then there is a codimension $2$ subvariety whose points are singular points of the foliation. 
Kupka singularities were generalized to arbitrary codimension by de~Medeiros in \cite{deMedTesis}, where stability for this singularities is proved in general. In codimension $q$ Kupka singularities come in subvarieties of codimension less or equal than $q+1$. In codimension $1$ there are many results relating the geometry of the variety of Kupka points with the global properties of the foliation, see \emph{e.g.}: \cite{omegar-ivan,omegar-kupka}. In higher codimension there is the work of Calvo-Andrade \cite{omegar2009}.

Another subject we look upon in this work is the study of the unfoldings of a foliation. Unfoldings in the context of foliations were introduced independently by Suwa and Mattei in different contexts, see \cite{suwa-review} for a survey on the subject. Unfoldings of foliations where computed mostly in some codimension $1$ cases, locally by Suwa (see \emph{loc. cit.}) and on $\PP^n$ by Molinuevo in \cite{moli}.

Recently we have related the study of unfoldings and singularities of a codimension $1$ foliation on $\PP^n$. Indeed, in \cite{mmq} we define a homogeneous ideal $I(\omega)$ defining a subscheme of the singular scheme of $\omega$ (see below for precise definitions), the elements of degree equal to the degree of $\omega$ in $I(\omega)$ are in natural correspondence with the infinitesimal unfoldings of $\omega$. Under generic conditions we can prove that if $K(\omega)$ is the ideal defining the closure of the variety of Kupka points then $\sqrt{I(\omega)}=\sqrt{K(\omega)}$, using this result we were able to compute the unfoldings of foliations of codimension $1$ on $\PP^n$ with split tangent sheaf and also prove the existence of Kupka points for every foliation in $\PP^n$ with reduced singular scheme.

\subsection*{Main results}

Our aim in this article is to generalize previous results on the relation of unfoldings and singular points of a foliation to arbitrary codimension and to foliations on a non-singular projective variety. 
In codimension $1$ there is a direct relation between unfoldings and a certain type of singularities which we call persistent singularities. In this respect we prove Proposition \ref{JJcIIcKK} relating Kupka and persistent singularities:
\begin{proposition*}
 Let $\JJ$ be the ideal sheaf of the singular locus of $\omega$, $\KK$ the ideal of the Kupka singularities of $\omega$ and $\II$ the ideal of persistent singularities. Then the following inclusions hold,
 \[
  \JJ\subseteq \II\subseteq \KK.
 \]
\end{proposition*}
\noindent and Theorem \ref{teoKnotempty} stating the existence of Kupka points under certain hypotheses:
\begin{theorem*}
 Let $X$ be a projective variety and $\LL\xrightarrow{\omega}\Omega^1_X$ a foliation of codimension $1$ such that $\JJ(\omega)$ is a sheaf of radical ideals and such that $c_1(\LL)\neq 0$ and $H^1(X,\LL)=0$. Then $\omega$ has Kupka singularities.
\end{theorem*}
In higher codimension the relation of persistent and Kupka singularities is not so clear, specially in the case where the foliation is not given locally by a complete intersection of $1$-forms, as in Example \ref{exampleCodimension2inA3}. However, under suitable cohomological conditions the absence of persistent singularities impose very strong consequences on the foliation. If $\EE$ is the sheaf of $1$-forms defining the foliation, and if $\ext^1_{\OO_X}(\EE, \mathrm{Sym}^2\EE)=0$ then the absence of persistent singularities implies the existence of a connection on $\EE$, see Theorem \ref{teoConnectionE}:
\begin{theorem*}
 Let $X$ be a projective variety and $\LL\xrightarrow{\omega}\Omega_X^q$ be an integrable $q$-form and $\EE$ be the associated subsheaf of $1$-forms $\EE\subseteq \Omega^1_X$.
 Let $\mathrm{Sym}^2(\EE)$ denote the symmetric power of $\EE$ and suppose $\ext^1_{\OO_X}(\EE,\mathrm{Sym}^2(\EE))=0$. If $\II(\omega)=\OO_X$ then $\EE$ admits a holomorphic connection, in particular is locally free and every Chern class of $\EE$ vanishes.
\end{theorem*}

\section{Kupka scheme in the Projective space for codimension 1 foliations}

Along this section we will revisit some definitions that we used in \cite{mmq}, among them we will define the Kupka variety  as a projective scheme $\kup$ over $\PP^n$ and $I=I(\omega)$ 
the \emph{ideal of persistent singularities} (\emph{a.k.a.}
\emph{unfoldings ideal}) of $\omega$. 
Then we will recall some results that we proved in \emph{loc. cit.} that we will generalize later. The scheme $\kup$ and the ideal $I$ were of central importance in those results. 
We refer the reader to \cite{mmq} for a full overview of this subject.

\

With the exception of Theorem \ref{teodivision} through this section we will restrict to the projective space $\PP^n$. So let us denote $S=\CC[x_0,\dots,x_n]$ to the homogeneous coordinate ring of $\PP^n$ and $\Omega^1_{\PP^n}(e)$ the sheaf of twisted differential 1-forms in $\PP^n$ of degree $e$. With $\sing_{set}$ we will denote the (set theoretic) singular set of $\omega\in H^0(\PP^n, \Omega^1_{\PP^n}(e))$ in $\PP^n$,
\[
 \sing_{set} = \{p\in \PP^n: \omega(p) = 0 \}\ .
\]

\

\begin{definition}\label{foliation} Let $\LL\simeq \OO_{\PP^n}(-e)$, $e\geq 2$, be a line bundle and $\omega:\LL\to \Omega^1_{\PP^n}$ be a morphism of sheaves, we will say that $\omega$ defines an
\emph{algebraic foliation of codimension 1} on $\PP^n$, if $\Omega^1_{\PP^n}/\LL$ is torsion free and the morfism is generated by a non zero global section
$\omega\in H^0(\PP^n,\Omega^1_{\PP^n}(e))$ such that
$\omega\wedge d\omega = 0$. 
We recall that such a foliation has \emph{geometric degree} $e-2$, where by geometric degree we mean the degree of annihilation of $\omega$ with a generic line immersed in $\PP^n$.
\end{definition}

\

The condition of $\Omega^1_{\PP^n}/\LL$ to be torsion free in the definition of a foliation is equivalent to ask the singular set to have codimension greater than 2. Indeed, this is the same to ask that $\omega$ is not of the form $f\cdot\omega'$, for some global section $f\in H^0(\PP^n,\OO_{\PP^n}(d))$
and a 1-form $\omega'\in H^0(\PP^n,\Omega^1_{\PP^n}(e-d))$. Also, integrable differential 1-forms define the same foliation up to scalar multiplication. Then, we will denote the set of codimension 1 foliations of geometric degree $e-2$ as
\begin{equation}\label{torsion}
\FF^1(\PP^n,e) := \left\{\omega\in\PP\left(H^0(\PP^n,\Omega^1_{\PP^n}(e))\right):\ \omega\wedge d\omega=0,\ \codim(\sing_{set})\geq 2 \right\}.
\end{equation}

\begin{definition}\label{IJ} We define the graded ideals of $S$ associated to $\omega$ as
\begin{align*}
I(\omega) &:= \left\{ h\in S:\  h\ d\omega = \omega\wedge\eta\text{ for some } \eta\in\Omega^1_S\right\}\\
J(\omega) &:= \left\{ i_X(\omega)\in S:\ X\in T_S \right\}.
\end{align*}
We will name $I(\omega)$ the \emph{ideal of persistent singularities} of $\omega$. We will also denote them $I=I(\omega)$ and $J=J(\omega)$ if no confusion arises.
\end{definition}

\begin{remark}\label{1notinI}
Notice that $1\not\in I$, since the class of $d\omega$ in the Koszul complex of $\omega$, $\mathcal{H}^2(\omega)$ is not zero, see Definition \ref{koszulcomplex}. Also $J(\omega)$ equals the ideal defining the singular locus of $\omega$. This last thing, can be seen by contracting with the vector fields ${\partial}/{\partial x_i}$. The definition given
for $J(\omega)$ is better 
suited for our schematic approach that we will develop next.
\end{remark}

\begin{definition}
For $\omega\in\FF^1(\PP^n,e)$, we define the \emph{Kupka set} as the subset of the singular set
\[
\mathpzc{K}_{set} = \overline{\{p\in \sing_{set} :  d\omega(p)\neq 0 \}}\ .
\]
\end{definition}

\begin{remark}
 Notice that the definition above it is not the standard definition of the Kupka set. Usually it is defined just as the set of points in $\sing_{set}$ such that $d\omega(p)\neq 0$. Instead, we consider the closure of that set.
\end{remark}

\begin{definition}\label{KK} For $\omega\in\FF^1(\PP^n,e)$, we define the \emph{Kupka scheme} $\kup$ as the scheme theoretic support of $d\omega$ at $\Omega^2_{S}\otimes_S 
(S\big/J)$. Then, $\kup=\proj(S/K(\omega))$ where $K(\omega)$ is the homogeneous ideal defined as
\[
K(\omega)=\ann(\overline{d\omega})+J(\omega)\subseteq S,\quad \overline{d\omega}\in \Omega^2_{S}\otimes_S \left(S\big/J(\omega)\right).
\]
We will denote $K=K(\omega)$ if no confusion arises.
\end{definition}

We recall the notion of \emph{ideal quotient} of two $S$-modules $M$ and $N$ as
\[
(N:M) := \left\{a\in S: a.M\subseteq N\right\},
\]
then, one could also define $K(\omega)$ as $K(\omega)=(J\cdot \Omega^2_S: d\omega)$. Also, given that $\Omega^2_S$ is free, we can also write
\begin{equation}\label{Kbis}
K(\omega)=(J(\omega):J(d\omega)),
\end{equation}
where $J(d\omega)$ denotes the ideal generated by the polynomial coefficients of $d\omega$.

\

From the properties of ideal quotient, it follows that if $J$ is radical,
then $K$ is radical as well.

\

With the Example 4.5 in \cite{mmq}[p.~1034] we showed that the algebraic geometric approach is indeed necessary, since the reduced structure associated to the Kupka scheme $\KK$ differs from the reduced variety associated to $\mathpzc{K}_{set}$.  With the following lemma we show that the Kupka scheme and the Kupka set coincide when the singular locus it is radical.

\begin{lemma}(\cite{mmq}[Lemma 4.6, p.1034])\label{K=Kset} 
Let $\omega\in\FF^1(\PP^n,e)$ such that $J=\sqrt{J}$. Then
\[
\kup = \mathpzc{K}_{set}.
\]
\end{lemma}

We have the following chain of inclusions, see Proposition \ref{JJcIIcKK} and Proposition \ref{JJcIIcKK2} for a generalization, in the codimension one and codimension $q$ case, respectively:

\begin{proposition}(\cite{mmq}[Proposition 4.7, p.~1035])\label{incl2}
Let $\omega\in\FF^1(\PP^n,e)$. Then, we have the following relations
\[
J\subseteq I\subseteq K\ .
\]
\end{proposition}

Let $\mathfrak{p}$ be a point in $\PP^n$, \emph{e.g.}, an homogeneous prime ideal in $S$ different from the \emph{irrelevant ideal} $(x_0.\ldots,x_n)$, and let $\omega$ be an integrable differential 1-form. We will denote with a subscript $\mathfrak{p}$ the localization at the point $\mathfrak{p}$ and with $\widehat{S}_\mathfrak{p}$ the completion of the local ring $S_\mathfrak{p}$ with respect to the maximal ideal defined by $\mathfrak{p}$.

\begin{definition} We say that $\mathfrak{p}\in\PP^n$ is a \emph{division point of $\omega$} if $1\in I(\omega)_\mathfrak{p}$.
\end{definition}

We now define a subset of the moduli space of foliations on which we are going to state our next result.

\begin{definition}\label{generic} We define the set $\UU\subseteq\FF^1(\PP^n,e)$ as
\[
\UU = \left\{\omega\in\FF^1(\PP^n,e): \ \forall \mathfrak{p}\not\in\kup,\,\mathfrak{p}\text{ is a division point of }\omega\right\}.
\]
\end{definition}

See Theorem \ref{propP=K} for a generalization of the following:

\begin{theorem}(\cite{mmq}[Theorem 4.12, p.~1036])\label{teo1}
Let $\omega\in\mathcal{U}\subseteq\FF^1(\PP^n,e)$. Then,
\[
\sqrt{I}=\sqrt{K}.
\]
Furthermore, if $\sqrt{I}=\sqrt{K}$ then $\omega\in\UU$.
\end{theorem}

See Theorem \ref{teoKnotempty} for a generalization of the following:

\begin{theorem}(\cite{mmq}[Theorem 4.24, p.~1041])\label{teo3}
Let $\omega\in\FF^1(\PP^n,e)$ such that $J=\sqrt{J}$.
Then
\[
\kup=\mathpzc{K}_{set}\neq \emptyset.
\]
\end{theorem}

\

The following statement is valid in a non-singular variety $X$ and we will use it later. We will consider a $1$-form $\omega$ on $X$ with singular set of codimension equal or greater than 2. And we will denote with $\mathcal{J}$ the ideal sheaf of $\sing$.


\begin{theorem}(\cite{mmq}[Theorem 2.7, p.~1030])\label{teodivision}
  Let $\omega$ be an integrable $1$-form in a non-singular variety $X$ and let $\pp\in\sing$ be such that $\JJ_\pp$ is radical and such that $d\omega_\pp\in \JJ_\pp \cdot \Omega^2_{X,\pp}$. Then there is a formal $1$-form $\eta$ such that $d\omega=\omega\wedge\eta$.
\end{theorem}

\section{Unfoldings over schemes}\label{unf-schemes}

Along this section we will give the definition of \emph{codimension $q$ foliation} on a smooth variety $X$. Then we will redefine the \emph{singular locus} with a scheme theoretic approach. Finally we define an \emph{unfolding} of a codimension $q$ foliation.

\

If $\Xi \in \Gamma(U, \bigwedge^p TX)$ is a multivector and $\varpi\in \Gamma(U,\Omega^q_X)$ a $q$-form we will denote by $i_\Xi \varpi \in \Gamma(U,\Omega^{q-p}_X)$ the contraction.
Recall that the \emph{Pl\"ucker relations} for $\varpi$ are given by
\[
 i_\Xi \varpi\wedge \varpi=0
\]
for any $\Xi\in\bigwedge^{q-1} TX$.

When $\varpi(\pp)\neq 0$ for some closed point $\pp\in X$ then $\varpi$ is \emph{locally decomposable} as a product $\varpi=\varpi_1\wedge\dots\wedge\varpi_q$ of $q$ $1$-forms.

\begin{definition}
Let $\LL$ be a line bundle and $\omega:\LL\to \Omega^q_X$, with $1\leq q\leq dim(X)-1$, be a (non trivial) morphism of sheaves, we will say that the morphism is \emph{integrable} if 
\begin{itemize}
\item $\Omega^q_X/\LL$ is torsion free.
\item The map 
\[
 i_\Xi \omega\wedge \omega:\LL\to\Omega^{q+1}_X\otimes\LL^{-1}
\]
is zero for every local section $\Xi$ of $\bigwedge^{q-1} TX$.
\item For every local section $s$ of $\LL$ and $\Xi$ of $\bigwedge^{q-1} TX$,
$\omega(s)$ verifies
\begin{equation}\label{frobenius2}
 d (i_\Xi \omega(s))\wedge \omega(s)=0.	
\end{equation}
\end{itemize}

We also say that $\omega$ determines a \emph{codimension $q$ foliation}.
\end{definition}

\begin{remark}
 By using Equation (\ref{frobenius2}) with $q=1$ we recover the definition of codimension one foliation as in Definition \ref{foliation}. 
\end{remark}

\begin{remark}
If $\omega(s)$ is locally decomposable for every $s\in X$ as a product of $q$ 1-forms $\varpi_1(s),\dots,\varpi_q(s)$ then there exist a rank $q$ vector bundle 
$\mathcal{E}\hookrightarrow \Omega^1_X$, locally generated by $\varpi_1(s),\dots,\varpi_q(s)$ 
and  such that $\LL\simeq \bigwedge^q\mathcal{E}$. 
 Reciprocally, given a locally free sheaf of rank $q$, $\mathcal{E}$ and a map $\mathcal{E}\hookrightarrow \Omega^1_X$, we have that $\bigwedge^q\mathcal{E}$ is a line bundle $\LL$ and a map $\LL\to\Omega^q_X$. The condition that $\Omega^q_X/\LL$ is torsion free is equivalent to $\Omega^1_X/\mathcal{E}$ being torsion free. 
Example \ref{exampleCodimension2inA3} shows that the condition \emph{locally free} is necessary for this equivalence.
\end{remark}

\begin{remark}\label{Edef}
Let $\omega:\LL\to\Omega^q$ be a integrable $q$-form. Then, we can consider
two maps,
\[
\xymatrix{
\bigwedge^{q-1} TX\otimes \LL\ar[r]^<<<<<{i_{(-)}\omega}
&\Omega^1_X\ar[r]^<<<<<{\omega\wedge -}&\Omega_X^{q+1}\otimes\LL^{-1}
}
\]
The integrability condition on $\omega$ implies that this diagram is a complex
and it is easy to check that its homology is supported over the points
where $\omega$ is not decomposable. We define the sheaf
associated to $\omega$, denoted $\EE=\EE(\omega)$, as the kernel 
of $\omega\wedge -$. By definition, $\EE$ is a reflexive sheaf.
\end{remark}

\begin{example}\label{exampleCodimension2inA3}
 Let $X=\mathbb{A}^3$ or, in the holomorphic case, a polydisk of dimension $3$. We take $v\in \Gamma(X,TX)$ a vector field, generic in the sense that in a coordinate system $(x_1,x_2,x_3)$ we can write $v=f_1\frac{\partial}{\partial x_1}-f_2\frac{\partial}{\partial x_2}+f_3\frac{\partial}{\partial x_3}$ with $f_1, f_2, f_3 \in k[x_1,x_2,x_3]$ and 
such that the ideal $(f_1, f_2,f_3)\subseteq k[x_1,x_2,x_3]$ is a complete intersection, that is, there are no nontrivial relations among the $f_i$'s.
 
 The vector field $v$ generates a codimension $2$ foliation in $X$, this foliation is determined by a $2$-form $\omega$ such that $i_v\omega=0$. 
 One such $\omega$ is given by
 \[
  \omega= f_3 dx_1\wedge dx_2 + f_2 dx_1\wedge dx_3 + f_1 dx_2\wedge dx_3.
 \]
It can be verified that this $\omega$ satisfies Pl\"ucker relations, is integrable, $i_v\omega=0$ and that $\Omega^2_X/(\omega)$ is torsion free. Therefore $\omega$ determines the same foliation of codimension $2$ as $v$. 
If we now look at the $1$-forms annihilated by $v$ we get the subsheaf generated by the forms
\[
\omega_1= f_3 dx_2+f_2dx_3,\quad \omega_2=f_3 dx_1- f_1 dx_3\ \text{ and } 
\omega_3=f_2 dx_1 + f_1 dx_2.
\]
 These generators satisfy the relation $f_1\omega_1 + f_2 \omega_2= f_3\omega_3$. The subsheaf $\mathcal{E}=(\omega_1,\omega_2,\omega_3)$ is generically of rank $2$ outside the zeros of the ideal $(f_1,f_2,f_3)$ but $\mathcal{E}\otimes k(\pp)$ is of rank $3$ when $\pp$ is in the zeros of this ideal. Therefore $\mathcal{E}$ is \emph{not} locally free.
Moreover when we compute the determinant of $\mathcal{E}$ we get $\wedge^2\mathcal{E}= (f_1,f_2,f_3)\cdot(\omega)\subseteq \Omega^2_X$,
\[
\omega_1\wedge\omega_2=f_3\omega,\quad
\omega_3\wedge\omega_1=f_2\omega,\quad
\omega_2\wedge\omega_3=f_1\omega.
\]
In particular $\omega$ is not in $\wedge^2\mathcal{E}$. But by \cite[Lemma, p.~210]{GH}, if $\omega$ is locally decomposable, then $\omega\in \wedge^2\mathcal{E}$.
Then $\omega$ is \emph{not} locally decomposable around the zeros of the ideal $(f_1,f_2,f_3)$.\hfill\qedsymbol
\end{example}

Composing a morphism $\omega:\LL\to \Omega^q$ with the contraction of forms with vector fields give us a morphism
\[
\bigwedge^q TX\otimes \LL\to\OO_X.
\]
\begin{definition}
The ideal sheaf $\JJ(\omega)$ is defined to be the sheaf-theoretic image of the morphism $\bigwedge^q TX\otimes \LL\to\OO_X$. The subscheme it defines is called \emph{the singular set} of $\omega$ and denoted $\sing\subseteq X$. We will denote it just as $\JJ$ if no confusion arises.
\end{definition}

\begin{remark}
 This definition agrees with Remark \ref{1notinI}, where we said that the ideal $J(\omega)$ gives the ideal defining the singular locus of $\omega$.
\end{remark}

From \cite{suwa-review}[(4.6) Definition, p.~192] we get the following definition for a codimension $q$ foliation:

\begin{definition}
Let $S$ be a scheme, $p\in S$ a closed point, and $\LL\xrightarrow{\omega} \Omega^q_X$ a codimension $q$ foliation on $X$. An \emph{unfolding} of $\omega$ is a codimension $q$ foliation $\widetilde{\LL}\xrightarrow{\widetilde{\omega}}\Omega^q_{X\times S}$ on $X\times S$ such that $\widetilde{\omega}|_{X\times\{p\}}\cong \omega$. In the case $S=\mathrm{Spec}(k[x]/(x^2))$ we will call $\widetilde{\omega}$ a \emph{first order infinitesimal unfolding}.
\end{definition}

\section{Kupka scheme in general for codimension 1 foliations}

Over this section we will restate the definition of \emph{persistent singularities} and of the \emph{Kupka scheme}, through its ideal sheaf, in a more general setting, see Definition \ref{II2} and Definition \ref{KK2}, respectively. 
In \cite{mmq} we showed that 
persistent singularities are related to unfoldings in codimension one.
We want to extend this relation to higher codimension.

First we prove Proposition \ref{JJcIIcKK}, generalizing Proposition \ref{incl2} in the codimension one case. Then we define the Kupka scheme and we prove Theorem \ref{propP=K} and Theorem \ref{teoKnotempty}, generalizing Theorems  \ref{teo1} and \ref{teo3}.

\

Given a line bundle $\LL$ and a global section $\omega\in H^0\left(X,\Omega^1_X\otimes \LL^{-1}\right)$ we will consider the Koszul complex associated with $\omega$,
\begin{equation}\label{koszulcomplex}
\xymatrix{K(\omega): & \OO_X\ar[r]^-{\wedge\omega}& \Omega^1_X\otimes\LL^{-1} \ar[r]^-{\wedge\omega} & \dots\ar[r] & \Omega^i\otimes\LL^{-i}\ar[r]& \dots}
\end{equation}
where we are following \cite[Chapter 2, B, p.~51]{GKZ} and using the identification $\bigwedge^k\left(\Omega^1_X\otimes \LL^{-1}\right)\simeq \left(\bigwedge^k \Omega^1_X \right)\otimes \left(\LL^{-k}\right)$. We will denote the cohomology sheaves of this complex by $\mathcal{H}^\bullet(\omega)$, the \emph{Koszul cohomology sheaves} of $\omega$.

\bigskip

We can use $K(\omega)$ to compute the codimension of $\sing$ by the well known result, see \cite[Theorem 17.4, p.~424]{eisenbud}:
\begin{theorem}\label{koszul}
 Let $\omega\in H^0\left(X,\Omega^1_X\otimes \LL^{-1}\right)$. The following 
statements are equivalent:
 \begin{itemize}
  \item[i)] $codim(\sing)\geq k$
  \item[ii)] $\mathcal{H}^\ell(\omega)=0$ for all $\ell <k$
 \end{itemize}
\end{theorem}

\begin{remark}\label{notation}
  Suppose now the morphism $\omega:\LL\to \Omega^1_X$ defines a foliation on $X$. 
  Given a trivializing open set $U$ and a choice of a trivialization $\OO_X|_U\cong \LL|_U$, we take a local generator $\varpi$ of $\LL(U)$ (we think about it as a $1$-form through the morphism $\LL\to \Omega^1_X$) and take the differential $d \varpi$. 
  This defines a $\CC$-linear morphism $\LL(U)\to \Omega^2_X(U)$, which in turn we can compose with the projection $\Omega^2_X\to \mathcal{H}^2(\omega)\otimes \LL^{\otimes 2}$.
  Note that the submodule $\OO_X(U)\cdot (d\varpi)$ of $\left(\mathcal{H}^2(\omega)\otimes\LL^{\otimes 2}\right)(U)$ is independent of the choice of the trivialization. In this way one gets a morphism of coherent sheaves,
  \[
    \LL\to \mathcal{H}^2(\omega)\otimes\LL^{\otimes 2}.
  \]
Or, equivalently, a (non trivial) global section of $\mathcal{H}^2(\omega)\otimes\LL$. 
We will denote the global section or the morphism indistinctly by $[d\omega]$. 
By Theorem \ref{koszul} above, we conclude that $codim(\sing)\geq 2$.
\end{remark}

\begin{definition}\label{II2} The subscheme of \emph{persistent singularities} of $\omega$ is the one defined by the ideal sheaf $\II(\omega):=\ann([d\omega])$, for $[d\omega]\in H^0\left(X,\mathcal{H}^2(\omega)\otimes\LL\right)$. We will denote it just as $\II$ if no confusion arises.
\end{definition}

\begin{remark}\label{remarkI}
  Let $\varpi\in \Omega^1_X(U)$ be a local generator of the image of $\LL\xrightarrow{\omega}\Omega^1_X$, then the local sections of $\II(\omega)$ in $U$ are given by 
  \[
    \II(U)=\{h\in \OO_X(U): \text{there is a section } \eta\in\Gamma(U,\Omega^1_X/\LL)\text{ s.t. }  hd \varpi=\varpi\wedge\eta\}. 
  \]
\end{remark}

\begin{remark}\label{formalremark}
  For a regular local ring $(R,\mathfrak{m})$, an $R$-module $M$ and an element $m\in M$, let us denote $\widehat{R}$ the $\mathfrak{m}$-adic completion of $R$ and $\widehat{M}=M\otimes\widehat{R}$. The element $m\otimes 1\in\widehat{M}$ has as annihilator the ideal $\ann(m)\otimes \widehat{R}$.
  Setting $R=\OO_{X,\pp}$, $M= \left(\mathcal{H}^2(\omega)\otimes\LL\right)_\pp$ and $m=[d\omega]_\pp$, and following the notation of Remark \ref{notation}, we have that 
  \[
   \ann([d\omega]_\pp\otimes 1)=\{h\in \widehat{\OO_{X,\pp}}: \text{there is a formal $1$-form } \eta \ \text{ s.t. } hd \varpi=\varpi\wedge\eta\}.
  \]
\end{remark}

\begin{proposition}
 Let $\pp \in X$ be a point in $\sing$, $\OO_{X,\pp}$ the local ring around $\pp$, and $X_\pp=\mathrm{Spec}(\OO_{X,\pp})$. Then $\pp$ is in the subscheme of persistent singularities if and only if for any infinitesimal first order unfolding $\widetilde{\omega}$  of $\omega$ in $X_\pp$, the point $(\pp, 0)\in X_\pp\times \mathrm{Spec}(k[\varepsilon]/(\varepsilon^2))$ is a singular point of $\widetilde{\omega}$. 
\end{proposition}
\begin{proof}
 Let $S=\mathrm{Spec}(k[\varepsilon]/(\varepsilon^2))$, $0\in S$ be its closed point, $p:X\times S\to S$ be the projection and $\iota:X\cong X\times\{0\}\hookrightarrow X\times S$ be the inclusion. 
 Then the sheaf $\Omega^1_{X\times S}$ can be decomposed as direct sum of $\iota_*(\OO_X)$-modules as 
 \[
  \Omega^1_{X\times S}\cong \iota_*(\Omega^1_X)\oplus \epsilon\cdot \iota_*(\Omega^1_X)\oplus \iota_*(\OO_X)d\epsilon.
 \]
A point $\pp\in\sing$ is \emph{not} a persistent singularity if and only if $1\in \II_\pp\subseteq \OO_{X,\pp}$ which, by Remark \ref{remarkI}, means that there is an open neighborhood $U\subseteq X$ of $\pp$, a local generator $\varpi$ of the image of $\LL\xrightarrow{\omega}\Omega^1_X$, and a section $\eta\in\Gamma(U,\Omega^1_X/\LL)$ such that $d\omega=\omega\wedge \eta$. 
By shrinking $U$ if necessary we can take a lifting of $\eta$ in $\Omega^1_X$ which by abuse of notation we also call $\eta$ and define
\[
 \widetilde{\omega}= \omega + \varepsilon \eta + d\varepsilon.
\]
Thus $\widetilde{\omega}$ is a form in $\Omega^1_{X\times S}$ and $\widetilde{\omega}(\pp,0)=
d\varepsilon\neq 0$, so $\pp\times \{0\}$ is \emph{not}  a singular point of $\widetilde{\omega}$.
Reciprocally, if there is an unfolding $\widetilde{\omega}$ of $\omega|_U$,
then
\[
 \widetilde{\omega}(\pp, 0)=\omega(\pp) + h(0) d\varepsilon.
\]
As $\pp$ is a singular point of $\omega$, we have $\omega(\pp)$, so if $(\pp,0)$ is not a singular point of $\widetilde{\omega}$, then $h(0)\neq 0$, so again shrinking $U$ if necessary we have that $h$ is a unit, hence $1\in \II_\pp$.
\end{proof}

Most of the known families of foliations on algebraic varieties present persistent singularities, see \cite{gmln,omegar,celn,pullback,fji,fmi,fj,mmq}. As it happens the absence of persistent singularities impose some restrictions on the line bundle $\LL$. To explain this we have to make explicit use of a result that is implied in the proof of Lefschetz Theorem on $(1,1)$ classes as is proved in \cite[Chapter 1.1 p.:~ 141]{GH}. 

\begin{lemma}
Let $\LL$ be a line bundle. Choose a trivialization $(U_i, \phi_i)_{i\in I}$ of $\LL$ with gluing data $g_{ij}\in \OO^*_X(U_{ij})$. The \v Cech  cocycle $\frac{1}{2\pi i }[d\log g_{ij}]\in Z^1(\Omega^1_X)$ represents the Chern class $c_1(\LL)$ of $\LL$ in $H^1(X,\Omega^1_X)$.
\end{lemma}

\begin{proof}
The claim follows from a careful reading of the proof of the Proposition in page 141 of \cite[Chern classes of line bundles, Chapter 1.1, p.~141]{GH}.
\end{proof}

\begin{proposition}\label{propc1L=0}
Let $X$ be a smooth projective variety over $\CC$.
If $\LL$ is a line bundle such that $H^1(X,\LL)=0$ and $\LL\xrightarrow{\omega}\Omega^1_X$ is a foliation without persistent singularities then $c_1(\LL)=0$, where $c_1(\LL)$ is the Chern class of the line bundle viewed in $H^2(X,\CC)$.
\end{proposition}

\begin{proof}
Let $(U_i,\phi_i)$ be a trivialization of $\LL$ with gluing data $g_{ij}\in \OO^*_X(U_{ij})$. On each $U_i$ we have a local generator of $\LL(U_i)$, namely $\phi_i^{-1}(1)$, we denote by $\omega_i$ the image under $\omega$ of this generator.
 The fact that the foliation defined by $\omega$ has no persistent singularities means that on each $U_i$ there is a local section $\eta_i$ of $\Omega^1_X/\LL(U_i)$
such that $d\omega_i=\omega_i\wedge\eta_i$.
On $U_{ij}$ the restriction of the local $1$-form $\omega_i$ satisfies
\[
\omega_i=g_{ij}\omega_j.
\]
So computing the de Rham differential of this forms on $U_{ij}$ gives us,
\begin{eqnarray*}
\omega_i\wedge\eta_i=d\omega_i=d(g_{ij}\omega_j)=\\
=g_{ij}d\omega_j+dg_{ij}\wedge \omega_j=\\
=g_{ij} \omega_j\wedge\eta_j + dg_{ij}\wedge \omega_j=\\
=g_{ij}\omega_j\wedge\left(\eta_j-\frac{dg_{ij}}{g_{ij}}\right).
\end{eqnarray*}
Subtracting both sides of the equality we get that, on $U_{ij}$, 
\[
\eta_i-\eta_j=\frac{dg_{ij}}{g_{ij}},
\]
as sections of $\Gamma(U_{ij},\Omega^1_X /\LL)$.
Therefore we get a \v{C}ech cochain $(\eta_i)_{i\in I}$ of \linebreak
$C^0(\Omega^1_X/\LL)$ whose border is 
\[
\partial (\eta)_{ij}=d\log g_{ij}\in B^1(\Omega^1/\LL).
\]
As the cocycle $(d\log g_{ij})\in Z^1(\Omega^1_X)$ represents $(2\pi i )c_1(\LL)$, the existence of the cochain $(\eta_i)$ implies $c_1(\LL)$ is in the kernel of the map $H^1(\Omega^1_X)\to H^1(\Omega^1_X/\LL)$ induced by the short exact sequence of sheaves
\[
0\to\LL\xrightarrow{\omega}\Omega^1_X\to \Omega^1_X/\LL\to 0.
\] 
The hypothesis $H^1(\LL)=0$ then implies $c_1(\LL)=0$.
\end{proof}

\begin{corollary}
Let $X$ be a smooth projective variety over $\CC$ such that every line bundle $\LL$ verifies $H^1(X,\LL)=0$ and such that $Pic(X)$ is torsion-free (\emph{e.g.}: $X$ smooth complete intersection). Then every foliation on $X$ have persistent singularities.
\end{corollary}
\begin{proof}
From the exponential sequence and the hypothesis $H^1(X,\OO_X)=0$ 
it follows that $c_1:Pic(X)\to H^2(X,\ZZ)$ is injective. Assume that 
$\omega:\LL\to\Omega^1_X$ is a foliation without persistent singularities. Then 
the above Proposition imply that $c_1(\LL)$ is a torsion element in $H^2(X,\ZZ)$.
But given that $Pic(X)$ is torsion-free, we get $\LL\cong\OO_X$.

In particular, $\omega$ is a global differential $1$-form
which contradicts the fact that $H^0(X,\Omega^1_X)=H^1(X,\OO_X)=0$.
\end{proof}

\begin{remark}\label{{domega}}
Given a trivializing open set $U$, a choice of a trivialization $\OO_X|_U\cong \LL|_U$ and a local generator $\varpi$ of $\LL(U)$, the mapping $\varpi\mapsto d\varpi$ defines a $\OO_X$-linear morphism $\LL\to \Omega^2_X\otimes \OO_{\sing}$. We will denote by $\{d\omega\}$ this morphism or equivalently the global section of $\Omega^2_X\otimes \OO_{\sing}\otimes\LL^{-1}$ it defines.
\end{remark}

\begin{definition}\label{KK2}The subscheme of \emph{Kupka singularities} of $\omega$ is the one defined by the ideal sheaf $\KK(\omega):=\ann(\{d\omega\})\in\Omega^2_X\otimes \OO_{\sing}\otimes\LL^{-1}$. We will denote it just as $\KK$ if no confusion arises.
\end{definition}

\begin{proposition}\label{JJcIIcKK}
 Let $\JJ$ be the ideal sheaf of the singular set of $\omega$, $\KK$ the ideal of the Kupka singularities of $\omega$ and $\II$ the ideal of persistent singularities. Then the following inclusions hold,
 \[
  \JJ\subseteq \II\subseteq \KK.
 \]
\end{proposition}
\begin{proof}
 Let $U\subseteq X$ be an open subscheme such that $\LL|_U\simeq\OO_X$, and $\varpi$ a local generator of $\LL(U)$.
 
 Suppose $h\in\JJ(U)\subseteq\OO_X(U)$ is a local section. By shrinking $U$ if necessary we may assume that there is a vector field $v\in T_X(U)$ such that $h= i_v(\omega)$. Then we have 
 \[
  0=i_v(\varpi\wedge d\varpi)=i_v(\varpi)d\varpi - \varpi \wedge i_v(d\varpi).
 \]
So, calling $\eta=i_v(d\varpi)$, we get $hd\varpi=\varpi\wedge \eta$.
Hence $h$ is in $\II(U)$, which proves the first inclusion. 

Now assume $h\in \II(U)$, then again by shrinking $U$ if necessary, we may assume that there is a $\eta\in \Omega^1_X/\LL (U)$ such that $hd\varpi= \varpi\wedge\eta$. By definition we have $\varpi \in \JJ(U)\cdot \Omega^1_X(U)$, then $hd\varpi\in \JJ(U)\cdot \Omega^2_X(U)$ so $h$ is in the annihilator of $\{d\omega\}$ in $\Omega^2_X\otimes \OO_{\sing}$. Then $h\in \KK(U)$, which proves the second inclusion.
\end{proof}

With the following results we can generalize Theorem \ref{teo1} and Theorem \ref{teo3} giving conditions for the existence of Kupka singularities:

\begin{definition}
Let $X$ be a smooth projective variety and $\LL\xrightarrow{\omega}\Omega^1_X$ a foliation of codimension $1$, we are going to call $\per\subseteq X$ the subschemes of persistent singularities.
\end{definition}

\begin{theorem}\label{propP=K}
Let $X$ be a smooth projective variety and $\LL\xrightarrow{\omega}\Omega^1_X$ a foliation of codimension $1$ such that $\JJ(\omega)$ is a sheaf of radical ideals. Let $\per\subseteq X$ and $\kup\subseteq X$ be the subschemes of persistent and Kupka singularities respectively. Then $\per_{\mbox{red}}=\kup_{\mbox{red}}$.  
\end{theorem}
\begin{proof}
We are going to prove that $X\setminus\per=X\setminus\kup$. By Proposition \ref{JJcIIcKK} we have $\kup\subseteq\per$, so $X\setminus \per\subseteq X\setminus \kup$. 
Now suppose $\pp$ is a point \emph{not} in $\kup$, by abuse of notation we will call $\omega$ a local generator of $\LL_\pp$ viewed as a $1$-form. As $\pp$ is not in $\kup$ 
then $d\omega \in \JJ_\pp\cdot\Omega^2_{X,\pp}$. By hypothesis $\JJ_\pp$ is radical and so by Theorem \ref{teodivision} we have that $d\omega$ decomposes as $\omega\wedge\eta$ for some formal $1$-form $\eta$, this implies $1\in \II_\pp$, so $\pp$ is not in $\per$ (see Remark \ref{formalremark}). 
\end{proof}

\begin{theorem}\label{teoKnotempty}
 Let $X$ be a smooth projective variety and $\LL\xrightarrow{\omega}\Omega^1_X$ a foliation of codimension $1$ such that $\JJ(\omega)$ is a sheaf of radical ideals and such that $c_1(\LL)\neq 0$ and $H^1(X,\LL)=0$. Then $\omega$ has Kupka singularities.
\end{theorem}
\begin{proof}
 This follows from Proposition \ref{propc1L=0} and Theorem \ref{propP=K}, as a foliation with $c_1(\LL)\neq 0$ and $H^1(X,\LL)=0$ has persistent singularities on one hand, and having radical singular ideal implies the reduced scheme defined by persistent singularities is equal to the reduced scheme of Kupka singularities, in particular this last scheme is not empty.
\end{proof}

\section{Infinitesimal unfoldings in codimension $q$}

Along this section we review the definition of unfolding of a codimension $q$ foliation on a variety $X$. We will also generalize the definitions of persistent singularities and of Kupka singularities for codimension $q$ foliations, see Definition \ref{defII} and Definition \ref{defKK}, respectively. We classify which singular points of $\omega$ are such that they extend to singular points of every unfolding $\widetilde{\omega}$, see Proposition \ref{prop}, 
and then, we generalize Proposition \ref{incl2} and Proposition \ref{JJcIIcKK} to the codimension $q$ case, see Proposition \ref{JJcIIcKK2}. 
Finally, with Theorem \ref{teoConnectionE} we establish that the absence of persistent singularities implies the existence of a connection on $\EE$, the sheaf of 1-forms
defining the foliation under strong cohomological assumptions.

\

Let $S=\mathrm{Spec}(k[\varepsilon]/(\varepsilon^2))$, $0\in S$ be its closed point, $p:X\times S\to S$ be the projection and $\iota:X\cong X\times\{0\}\hookrightarrow X\times S$ be the inclusion. 
Then the sheaf $\Omega^q_{X\times S}$ can be decomposed as direct sum of $\iota_*(\OO_X)$-modules as 
 \[
  \Omega^q_{X\times S}\cong \iota_*\Omega^q_X\oplus \varepsilon\cdot (\iota_*\Omega^q_X)\oplus \iota_*\Omega^{q-1}_X\wedge d\varepsilon.
 \]
Given a codimension $q$ foliation determined by a morphism $\LL\xrightarrow{\omega}\Omega^q_X$, and a first order infinitesimal unfolding $\widetilde{\omega}: \widetilde{\LL}\to \Omega^q_{X\times S}$ of $\omega$, we take local generators $\varpi$ of $\LL(U)$ and $\widetilde{\varpi}$ of $\widetilde{\LL}(U\times S)$. Suppose $\omega$ and $\widetilde{\omega}$ are locally decomposable, then we may take $U$ small enough such that $\varpi$ and $\widetilde{\varpi}$ decompose as products 
\[
 \varpi=\varpi_1\wedge\dots\wedge\varpi_q, \qquad \widetilde{\varpi}=\widetilde{\varpi}_1\wedge\dots\wedge \widetilde{\varpi}_q.
\]
Then we can write $\widetilde{\varpi}_i=\varpi_i + \varepsilon \eta_i + h_i d\varepsilon$
and the equations $d\widetilde{\varpi}_i\wedge \widetilde{\varpi}=0$ for $i=1,\dots, q$ 
are equivalent to the equations

\begin{equation*}
 \left\{
\begin{aligned}
&d\eta_i\wedge \varpi+ d\varpi_i\wedge \left(\sum_{j=1}^q (-1)^j\eta_j \varpi_{\widehat{j}}\right)=0,\qquad (i=1,\dots,q), \\
&(dh_i-\eta_i)\wedge \varpi+ d\varpi_i\wedge \left(\sum_{j=1}^q (-1)^jh_j \varpi_{\widehat{j}}\right)=0,\qquad (i=1,\dots,q), 
\end{aligned}
\right.
\end{equation*}

where $ \varpi_{\widehat{j}}=\varpi_1\wedge\dots\wedge\varpi_{j-1}\wedge\varpi_{j+1}\wedge\dots\wedge\varpi_q\in \Omega^{q-1}_X(U)$.

\

As is shown in \cite[proof of (6.1) Theorem, p.~199]{suwa-review} the second equation implies the first. So we finally get that the equations $d\widetilde{\varpi}_i\wedge \widetilde{\varpi}=0$ for $i=1,\dots, q$ are equivalent to 

\begin{equation}\label{equnfcodq}
\begin{aligned}
 \left\{
 (dh_i-\eta_i)\wedge \varpi+ d\varpi_i\wedge \left(\sum_{j=1}^q (-1)^jh_j \varpi_{\widehat{j}}\right)=0,\qquad (i=1,\dots,q)\ .
\right. 
\end{aligned}
\end{equation}

\begin{proposition}\label{prop}
 Suppose $\pp$ is a singular point of $\omega$. Then there exist an infinitesimal unfolding $\widetilde{\omega}$ of $\omega$ in $X_\pp$ such that $(\pp,0)$ is \emph{not} a singular point of $\widetilde{\omega}$ if and only if $\omega$ is decomposable locally around $\pp$, not all  $\varpi_{\widehat{j}}(\pp)$ vanish and there are $1$-forms $\alpha_{ij}\in\Omega^1_{X,\pp}$ for $i,j=1,\ldots,q$ such that 
 \[
  d\varpi_i=\sum_{j=1}^q \alpha_{ij}\wedge \varpi_j,\qquad \text{for }i=1,\dots,q.
 \]
\end{proposition}
\begin{proof}
 Given local forms $\alpha_{ij}\in \Omega^1_{X,\pp}$ such that $d\varpi_i=\sum_{j=1}^q \alpha_{ij}\wedge \varpi_j, (i=1,\dots,q)$ we may take local sections $h_i\in\OO_{X,\pp}$ such that $\sum_{i=1}^q(-1)^ih_i(\pp)\varpi_{\widehat{i}}(\pp)\neq 0$.
 With that choice of $h_i$'s we take $\eta_i:=dh_i+\sum_{j=1}^q (-1)^j h_j\alpha_{ij}$.
 We will see that the $\eta_i$'s and $h_i$'s determine an unfolding of $\omega$ locally around $\pp$.
 For that we need to verify the Equation (\ref{equnfcodq}) above. Indeed we have
 \begin{align*}
  (dh_i-\eta_i)&\wedge\varpi+d\varpi_i\wedge\left(\sum_{j=1}^q (-1)^jh_j\varpi_{\widehat{j}}\right)=\\
 &= (dh_i-\eta_i)\wedge\varpi+\left( \sum_{k=1}^q\alpha_{ik}\wedge \varpi_k \right)\wedge\sum_{j=1}^q (-1)^j\varpi_{\widehat{j}}=\\
 &=(dh_i-\eta_i)\wedge\varpi + \left(\sum_{j,k=1}^1(-1)^j h_j\alpha_{ik}\wedge\varpi_k\wedge\varpi_{\widehat{j}}\right)=\\
 &=(dh_i-\eta_i)\wedge \varpi + \left(\sum_{j=1}^q (-1)^j\alpha_{ij}\wedge \varpi\right) =\\
 &=\left((dh_i-\eta_i)+\sum_{j=1}^q (-1)^j h_j\alpha_{ij} \right)\wedge \varpi .
 \end{align*}
 And from the definition of the $\eta_i$ we have that
 \begin{align*}
 &\left((dh_i-\eta_i)+\sum_{j=1}^q (-1)^j h_j\alpha_{ij} \right)\wedge \varpi =\\
 &= \left(-\sum_{j=1}^q (-1)^j h_j\alpha_{ij} +\sum_{j=1}^q (-1)^j h_j\alpha_{ij} \right)\wedge \varpi= 0\\ 
 \end{align*}

Then we have an unfolding $\widetilde{\omega}$ given locally around $\pp$ by 
\[
 \bigwedge_{i=1}^q (\varpi_i+\varepsilon \eta_i + h_i d\varepsilon)= \varpi+\varepsilon \left(\sum_{i=1}^q \eta_i\wedge \varpi_{\widehat{j}}\right)+ \left(\sum_{j=1}^q (-1)^jh_j \varpi_{\widehat{j}}\right)\wedge d\varepsilon.
\]
As $\sum_{j=1}^q (-1)^jh_j \varpi_{\widehat{j}}\neq 0$ then $\widetilde{\omega}$ does not vanishes on $(\pp,0)$.

\bigskip

Reciprocally, let us suppose there is an unfolding $\widetilde{\omega}$ such that $\widetilde{\omega}(\pp,0)\neq 0$.
As $\widetilde{\omega}$ satisfies Pl\"ucker relations and does not vanish in $\pp$, 
then it decomposes as a product of $1$-forms $\varpi_i+\varepsilon\eta_i+h_id\varepsilon$, $i=1,\dots,q$.
As $\widetilde{\omega}|_{X\times\{0\}}=\omega$ then any local generator $\varpi$ of the image of $\omega$ is locally decomposable as $\varpi_1\wedge\dots\wedge\varpi_q$.
We want to prove that the class $[d\varpi_i]$ of $d\varpi_i$ in $\Omega^2_{X,\pp}/((\varpi_1,\dots,\varpi_q)\wedge\Omega^1_{X,\pp})$ is zero for $i=1,\dots, q$.
Let $\mathfrak{q}$ be a point in the support of $[d\varpi]$, then $\omega$ is singular in $\mathfrak{q}$, for otherwise $[d\varpi]=0$ because of the Frobenius condition $d\varpi_i\wedge \varpi = 0$, for $i=1,\ldots,q$. 
By Equation (\ref{equnfcodq}), we have $\sum_{j=1}^q (-1)^jh_j \varpi_{\widehat{j}}(\pp)\neq 0$, in particular not all of the $\varpi_{\widehat{j}}(\pp)$ vanishes. 
Without any loss of generality, we may assume $\varpi_{\widehat{1}}(\pp)$ does not vanish. Then also $\varpi_{\widehat{1}}(\mathfrak{q})\neq 0$.
But $\varpi(\mathfrak{q})=0$, therefore $\varpi_2(\mathfrak{q}),\dots,\varpi_q(\mathfrak{q})$ are linearly independent and $\varpi_1(\mathfrak{q})$ is a linear combination of them.
Hence $\varpi_{\widehat{j}}(\mathfrak{q})=f_j \varpi_{\widehat{1}}(\mathfrak{q})$. 
Then evaluating Equation (\ref{equnfcodq}) in $\mathfrak{q}$, and adding the term $h_{1}(\mathfrak{q}) d\varpi_i(\mathfrak{q}) \wedge \varpi_{\widehat{1}}(\mathfrak{q})$, gives
\[
h_{1}(\mathfrak{q}) d\varpi_i(\mathfrak{q}) \wedge \varpi_{\widehat{1}}(\mathfrak{q})=(dh_i-\eta_i)\wedge\varpi(\mathfrak{q}) +\left(\sum_{j=2}^q h_j(\mathfrak{q}) f_j(\mathfrak{q})\right)d\varpi_i(\mathfrak{q})\wedge\varpi_{\widehat{1}}(\mathfrak{q}) \ .
\]
So, after clearing $h_1(\mathfrak{q})\neq 0$, there is a $1$-form $\alpha_{i1}$ such that 
\[
d\varpi_i\wedge \varpi_{\widehat{1}}(\mathfrak{q})= \alpha_{i1}\wedge \varpi_1\wedge \varpi_{\widehat{1}}(\mathfrak{q}) ,
\]
then we have $(d\varpi_i-\alpha_{i1}\wedge\varpi_1)\wedge \varpi_{\widehat{1}}(\mathfrak{q})=0$, but as $\varpi_{\widehat{1}}(\mathfrak{q})\neq 0$ , this implies that there are forms $\alpha_{ij}$ such that
\[
 (d\varpi_i-\alpha_{i1}\wedge\varpi_1)(\mathfrak{q})=\sum_{j\neq 1}\alpha_{ij}\wedge \varpi_j(\mathfrak{q}).
\]
Hence $[d\varpi_i]=0$ in any point of its support, a contradiction, so $[d\varpi_i]=0$ in $\Omega^2_{X,\pp}/((\varpi_1,\dots,\varpi_q)\wedge\Omega^1_{X,\pp})$.
\end{proof}

 Let $\LL\xrightarrow{\omega}\Omega^q_X$ be an integrable morphism determining a subsheaf $\mathcal{E}\to\Omega^1_X$. Composing $\omega$ with wedge product gives a morphism
$\LL\otimes\Omega^2_X\xrightarrow{} \Omega^{q+2}_X$ and, tensoring by $\mathcal{L}^{-1}$, we get a morphism $\Omega^2_X\xrightarrow{} \Omega^{q+2}_X\otimes\mathcal{L}^{-1}$ which we will call $(\omega\wedge -)_{\Omega^2_X}$ to remark that the domain is $\Omega^2_X$. 
As $\omega$ is integrable, following Remark \ref{Edef} we get a morphism  $\xymatrix@1{\mathcal{E}\otimes\Omega^1_X \ar@{^(->}[r] & \Omega^1_X\otimes \Omega^1_X \ar[r]^{} & \Omega^2_X}$. Then we have that the sheaf $\mathcal{E}\otimes\Omega^1_X$ is in the kernel of $(\omega\wedge -)_{\Omega^2_X}$, since the following diagram commutes
\[
 \xymatrix@C=50pt{
 \Omega^1_X\otimes \Omega^1_X \ar[r]^-{-\wedge -} \ar[d]_-{Id \otimes(\omega\wedge -)} & \Omega^2_X \ar[d]^-{(\omega\wedge -)_{\Omega^2_X}}\\
\Omega^1_X\otimes\Omega^{q+1}_X\otimes\mathcal{L}^{-1}\ar[r]_-{(-\wedge -)\otimes Id}& \Omega^{q+2}_X\otimes\mathcal{L}^{-1}.
}
\]

This allow us to give the following definition.

\begin{definition}
  Let $\LL\xrightarrow{\omega}\Omega^q_X$ be an integrable morphism, we define the sheaf $\mathcal{H}^2(\omega)$ as
 \[
  \mathcal{H}^2(\omega):= \ker((\omega\wedge -)_{\Omega^2_X}) /\mathcal{E}\otimes \Omega^1_X.
 \]
\end{definition}

\begin{remark}
 The restriction of the de Rham differential to $\mathcal{E}$ gives a sheaf map $\mathcal{E}\to \Omega^2_X$ which is not $\OO_X$-linear but whose image is in $ \ker((\omega\wedge -)_{\Omega^2_X}) $ as $\omega$ is integrable.
 The projection of this map to $\mathcal{H}^2(\omega)$ is however $\OO_X$-linear as $d g\varpi\cong g d\varpi \mod \mathcal{E}\otimes\Omega^1_X$ for every local section $\varpi$ of $\mathcal{E}$. 
\end{remark}

 Let us fix $\LL\xrightarrow{\omega}\Omega^q_X$ be an integrable morphism determining a subsheaf $\mathcal{E}\hookrightarrow
\Omega^1_X$. Then we have the following definitions:

\begin{definition}\label{defII}
The subscheme of \emph{persistent singularities} of $\omega$ is the one defined by the ideal sheaf $\II(\omega)$ to be the annihilator of $d(\mathcal{E})$ in $\mathcal{H}^2(\omega)$. In other words the local sections of $\II(\omega)$ in an open set $U\subseteq X$ are given by
 \begin{align*}
  \II(\omega)(U)&=\left\{ h\in \OO_X(U): \forall \varpi\in\mathcal{E}(U),\ hd\varpi=\sum_{j=1}^q \alpha_j\wedge\omega_j \right. \text{ for some local }\\
  &\hspace{3.7cm}\text{$1$-forms $\alpha_j\in\Omega^1_X(U)$ and forms $\omega_j\in\mathcal{E}(U)$ }
 \Bigg\}\ . 
 \end{align*}
We will denote it just as $\II$ if no confusion arises.
\end{definition}

\begin{example}
 With the following example we are showing that the ideal $\II(\omega)$ can have codimension greater than $2$. Let us consider the $2$-form in $\PP^3$ defined by $\omega = i_{\frac{\partial}{\partial x_0}}(i_R\Omega)$ where $\Omega=dx_0\wedge dx_1\wedge dx_2 \wedge dx_3$ and $R$ denotes the radial vector field $\sum_{i=0}^3 x_i\frac{\partial}{\partial x_i}$. We get that:
\begin{align*}
 \omega &= -x_3 \ dx_1\wedge dx_2+x_2 \ dx_1\wedge dx_3-x_1\ dx_2\wedge dx_3
\end{align*}
Such  a differential form it is locally decomposable and locally integrable and has singular locus of codimension 3. The ideal of persistent singularities has also codimension 3 and it coincides with the ideal of the singular locus. This can be easily seen since the singular locus  are all Kupka points. We suggest to use the software {\tt DiffAlg}, see \cite{diffalg} for more elaborate computations.
\end{example}

We can consider an extension of Remark \ref{{domega}} for $\omega\in\Omega^q_X$. Then:

\begin{definition}\label{defKK}
The subscheme of \emph{Kupka singularities} of $\omega$ is the one defined by the ideal sheaf  $\KK(\omega):= \ann(\{d\omega\})\in \Omega^{q+1}_X\otimes\OO_{\sing}\otimes \LL^{-1}$.  We will denote it just as $\KK$ if no confusion arises.
\end{definition}

\begin{remark}
We would like to notice that both definitions above coincide to the ones given in the codimension 1 case, as the reader can see by comparing them to Definition \ref{II2} and Remark \ref{remarkI} and to Defintion \ref{KK2}, respectively.
\end{remark}

\begin{lemma}\label{pluckerFiltration}
 Given a short exact sequence of modules 
 \[
  0\to M\to P\to N \to 0,
 \]
there is a filtration in $\bigwedge^q P$.
\[
 \bigwedge^q P=F^0 \supseteq F^1\supseteq \dots \supseteq F^{q+1}=(0),
\]
such that 
\[
 F^i/F^{i+1}\cong \bigwedge^{q-i}N\otimes \bigwedge^{i}M.
 \]
\end{lemma}

\begin{proof}
 The result follows from defining $F^i\subseteq \bigwedge^q P$ to be the submodule generated by the elements of the form $(m_1\wedge\dots\wedge m_i\wedge a_{i+1}\wedge\dots\wedge a_q)$ where $m_j\in M$.
\end{proof}

\begin{proposition}\label{JJcIIcKK2}
 Given an integrable morphism $\LL\xrightarrow{\omega}\Omega^q_X$ we have the inclusions $\JJ(\omega)\subseteq\II(\omega)$ and $\JJ(\omega) \subseteq \KK(\omega)$.
 If moreover $\omega$ is locally decomposable (\emph{i.e.} if $\EE$ is locally free) then we have $\JJ(\omega)\subseteq\II(\omega)\subseteq \KK(\omega)$.
\end{proposition}
\begin{proof}
 To ease the notation let us set $\JJ=\JJ(\omega)$, and likewise with $\II$ and $\KK$.
 Let $h$ be a local section of $\JJ$, and by abuse of notation we will call $\omega$ a local generator of the image of the morphism $\omega:\LL\to\Omega^q_X$, then by definition of $\JJ$ there is a local $q$-vector $v\in \bigwedge^{q}T_X$ such that $h=i_v\omega$. Then taking the filtration $\Omega^2_X=F^0\supseteq F^1\supseteq F^2\supseteq F^3=0$ of Lemma \ref{pluckerFiltration} associated to the exact sequence
 \[
  \xymatrix{
  0 \ar[r] & \mathcal{E} \ar[r] & \Omega^1_X\ar[r] & \Omega^1_X\big/\mathcal{E} \ar[r] & 0
  }
 \]
 for $q=2$, we can say that $h$ is in $\II$ if and only if for every local section $\varpi \in \EE$ we have $hd\varpi\in F^1\simeq\mathcal{E}\wedge\Omega^1_X$. To establish this we recall that for every local section $\varpi$ of $\EE$ the equation $d\varpi\wedge\omega=0$ holds. Then contracting with $v$ we get 
 \begin{align*}
  0&= i_v(d\varpi\wedge\omega)=d\varpi\wedge i_v\omega \ +\\
  &\hspace{1cm}+ \sum_{\substack{a_j\in T_X,\ b_j\in \bigwedge^{q-1}T_X\\ a_1\wedge b_1+\dots +a_r\wedge b_r=v\\j=1}}^r i_{a_j} d\varpi\wedge i_{b_j} \omega+ \sum_{\substack{c_j\in \bigwedge^2 T_X,\ d_j\in \bigwedge^{q-2}T_X\\ c_1\wedge d_1+\dots +c_r\wedge d_s=v\\j=1}}^r i_{c_j} d\varpi \wedge i_{d_j}\omega.
 \end{align*}
 To verify that $hd\varpi = d\varpi\wedge i_v\omega \in F^1$ we can see that the last two summands of the above equation are in $F^1$.
 By definition of $\EE$ we have that $i_b \omega$ is a local section of $\EE$, so every summand of the form $i_a d\varpi \wedge i_b \omega$ is in $\Omega^1_X\wedge\EE$.
 Hence, to see that $hd\omega\in F^1$ it suffices to show that $i_d\omega$ is in $F^1$ for every $d\in \bigwedge^{q-2}T_X$. To see this we can calculate the class of $i_d\omega$ in $\Omega^2_X/F^1=F^0/F^1\cong \bigwedge^2(\Omega^1_X/\EE)$. The dual sheaf $(\Omega^1_X/\EE)^\vee \subseteq T_X$ is the distribution defined by $\omega$, that is, is the sheaf of vector fields $V$ such that $i_v\omega=0$. Then, when we evaluate $i_d \omega$ in a section $v_1\wedge v_2\in \bigwedge^2(\Omega^1_X/\EE)^\vee$ we get $0$. As $\bigwedge^2(\Omega^1_X/\EE)$ is torsion-free then the class of $i_d \omega$ in $\Omega^2_X/F^1$ is zero, then $i_d\omega\in F^1$, which means $h d\varpi$ is in $F^1$ as we wanted to show.
 
 \medskip
 
 The second assertion is clear by definition, as $\KK$ is the annihilator of a section whose support is contained in $\sing$.
 
 \medskip
 
 Now suppose $\EE$ is locally free. So we can take local generators $\varpi_1,\dots,\varpi_q$ of $\EE$, this sections verify that $\omega=\varpi_1\wedge\dots\wedge\varpi_q$.
 Then for every section $h$ of $\II$ there are local $1$-forms $\alpha_{ij}\in\Omega^1_X$ such that
 \[
  hd\varpi_i=\sum_{j=1}^q \alpha_{ij}\wedge\varpi_j.
 \]
 Therefore we have
 \begin{align*}
  hd\omega&=d(\varpi_1\wedge\dots\wedge\varpi_q)=\sum_{i=1}^q (-1)^i \varpi_1\wedge\dots\wedge d\varpi_i \wedge \varpi_{i+1}\wedge\dots\wedge\varpi_q=\\
  &=\sum_{i=1}^q (-1)^i \varpi_1\wedge\dots\wedge \left(\sum_{j=1}^q \alpha_{ij}\right)\wedge\varpi_j \wedge \varpi_{i+1}\wedge\dots\wedge\varpi_q=\\
&=\sum_{i=1}^q \alpha_{ii}\wedge \varpi_1\wedge\dots\wedge\varpi_q=\left(\sum_{i=1}^q \alpha_{ii}\right)\wedge\omega.
 \end{align*}
 In particular $hd\omega$ vanishes in $\sing$ so $h$ is in $\KK$. 
\end{proof}

\begin{example}
 Let $\omega\in \Omega^2_{\mathbb{A}^3}$ be like in Example \ref{exampleCodimension2inA3} so we write 
 \[
  \omega= f_3 dx_1\wedge dx_2+f_2 dx_1\wedge dx_3+f_1 dx_2\wedge dx_3.
 \]
So we have 
 \[
  d\omega = \left(\frac{\partial f_3 }{\partial x_3}- \frac{\partial f_2 }{\partial x_2}+ \frac{\partial f_1 }{\partial x_1}\right) dx_1\wedge dx_2\wedge dx_3.
 \]
 For a general choice of the $f_i$'s the restriction $d\omega|_{\sing}$ does not vanish, so $\JJ=\KK$. 
 
 However, by setting for instance $f_3=f_3(x_1,x_2)$, $f_2=f_2(x_1,x_3)$ and $f_1=f_1(x_2,x_3)$, we get a form $\omega$ such that $d\omega=0$. 
 With this choice of $\omega$ we have $\KK=\OO_X$. When computing the ideal $\II$ for this case we need to check that $h d\omega_i= \alpha_{i1}\wedge\omega_1+\alpha_{i2}\wedge\omega_2+\alpha_{i3}\wedge\omega_3$ for $i=1,2,3$, where the $\omega_i$'s are the generators of $\EE$ of Example \ref{exampleCodimension2inA3}
and $h\in \OO_X$.
Further specializing our choice of $\omega$ we can take $f_3=x_1$ and $f_2=x_1+x_3$, in order to get $d\omega_1=dx_1\wedge dx_2 + dx_1\wedge dx_3$, so clearly $1\notin \II(\omega)$.
 
 So we see that there are cases where $\KK=\OO_X$ and $1\notin \II$. This is in stark contrast to the situation in codimension $1$ where, from Theorem \ref{teodivision}, follows that the condition $\JJ=\sqrt{\JJ}$ implies $\sqrt{\II}=\sqrt{\KK}$.\hfill\qedsymbol
\end{example}

Now we present a generalization of Proposition \ref{propc1L=0} to arbitrary codimensions. 

\begin{theorem}\label{teoConnectionE}
 Let $X$ be a projective variety and $\LL\xrightarrow{\omega}\Omega_X^q$ be an integrable $q$-form and $\EE$ be the associated subsheaf of $1$-forms $\EE\subseteq \Omega^1_X$.
 Let $\mathrm{Sym}^2(\EE)$ denote the symmetric power of $\EE$ and suppose $\ext^1_{\OO_X}(\EE,\mathrm{Sym}^2(\EE))=0$. If $\II(\omega)=\OO_X$ then $\EE$ admits a holomorphic connection, in particular is locally free (in other words the foliation is locally decomposable) and every Chern class of $\EE$ vanishes.
\end{theorem}
\begin{proof}
 In order to prove the vanishing of the Chern classes of $\EE$ we are going to use Atiyah's classical result \cite[Theorem 4, p.~192]{atiyah} which states that if a holomorphic vector bundle on a compact K\"ahler manifold admits a holomorphic connection, then its Chern classes are all zero.
 We will then produce a holomorphic connection for $\EE$ in this case.
 The condition $\II(\omega)=\OO_X$ implies that for every local section $\varpi$ of $\EE$ we have $d\varpi=\sum_i\alpha_i\wedge\omega_i$ for some local $1$-forms $\alpha_i$ and $\omega_i \in \EE$. In other words, let $F^\bullet$ be the filtration of $\Omega^2_X$ associated with the short exact sequence 
 \[
  0\to \EE\to\Omega^1_X\to \Omega^1_X/\EE\to 0,
 \]
 as in Lemma \ref{pluckerFiltration},by the proof of this lemma the subsheaf $F^1$ is the image of exterior multiplication $\Omega^1_X\otimes\EE\to \Omega^2_X$. Then the de~Rham differential applied to sections of $\EE$ give us a map $d:\EE\to F^1\subseteq \Omega^2_X$ such that $d(f\varpi)=df\wedge\varpi+fd\varpi$, that is a differential operator of order $1$ between $\EE$ and $F^1$. We will call $\mathrm{Diff}^{\leq 1}(A, B)$ the set of differential operators of order $\leq 1$ between two sheaves $A$ and $B$.
 Let us denote with $\mathcal{PE}$ the sheaf of principal parts of $\EE$ of order $1$, see \cite[16.7, p. 36]{egaivIV}, this sheaf is defined by the universal property $\hom_{\OO_X}(\mathcal{PE},M)=\mathrm{Diff}^{\leq 1}(\EE, M)$ for every coherent sheaf $M$, see \cite[Proposition 16.8.4, p. 41]{egaivIV}. So the de~Rham differential defines a coherent sheaves morphism $[\nabla]:\mathcal{PE}\to F^1$. To see if we can lift $[\nabla]$ to a morphism $\nabla:\mathcal{PE}\to \Omega^1_X\otimes\EE$ defining a connection, we first observe that the kernel of the map $\Omega^1_X\otimes\EE\to F^1$ (which is the exterior multiplication map) contains $\mathrm{Sym}^2(\EE)$ and is contained in $\EE\otimes\EE$, so the kernel must be $\mathrm{Sym}^2(\EE)$. Then we have the short exact sequence $0\to\mathrm{Sym}^2(\EE)\to\Omega^1_X\otimes \EE\to F^1\to 0$ which gives an exact sequence of modules
 \begin{align*}
  0 &\to \hom_{\OO_X}(\mathcal{PE},\mathrm{Sym}^2(\EE))\to \hom_{\OO_X}(\mathcal{PE},\Omega^1_X\otimes\EE)\to \hom_{\OO_X}(\mathcal{PE},F^1 )\xrightarrow{\delta} \\
  &\xrightarrow{\delta} \ext^1_{\OO_X}(\mathcal{PE},\mathrm{Sym}^2(\EE))\to \ext^1_{\OO_X}(\mathcal{PE},\Omega^1_X\otimes\EE)\to\cdots
 \end{align*}
 So $[\nabla]$ lifts to a morphism $\mathcal{PE}\to \Omega^1_X\otimes\EE$ if and only if is in the kernel of $ \hom_{\OO_X}(\mathcal{PE},F^1)\xrightarrow{\delta} \ext^1_{\OO_X}(\mathcal{PE},\mathrm{Sym}^2(\EE))$.
 
 In order to compute $\ext^1_{\OO_X}(\mathcal{PE},\mathrm{Sym}^2(\EE))$ recall the short exact sequence of sheaves
 \[
  0\to \Omega^1_X \to \mathcal{P} \to \OO_X \to 0,
 \]
 tensoring with $\EE$ this gives the sequence 
 \[
  0\to \Omega^1_X\otimes\EE \to \mathcal{PE} \to \EE \to 0
 \]
 (notice that the first term in the left is the sheaf $Tor^X_1(\EE,\OO_X)$ which is $0$ as $\OO_X$ is flat over $\OO_X$). 
 The last exact sequence give rise to an exact sequence
\begin{align*}
   \cdots\to \ext^1_{\OO_X}(\EE,\mathrm{Sym}^2(\EE))&\to \ext^1_{\OO_X}(\mathcal{PE},\mathrm{Sym}^2(\EE))\to \\
  &\to\ext^1_{\OO_X}(\Omega^1_X\otimes\EE,\mathrm{Sym}^2(\EE))\to \cdots
\end{align*}
 Recall that the group $\ext^1_{\OO_X}(\mathcal{PE},\mathrm{Sym}^2(\EE))$ can be regarded as the group of isomorphism classes of extensions of $\mathcal{PE}$ by $\mathrm{Sym}^2(\EE)$. Viewed like this, the morphism $\delta: \hom_{\OO_X}(\mathcal{PE},F^1 )\to  \ext^1_{\OO_X}(\mathcal{PE},\mathrm{Sym}^2(\EE))$ evaluated at an element $a\in \hom_{\OO_X}(\mathcal{PE},F^1)$ returns the isomorphism class of the extension $0 \to \mathrm{Sym}^2(\EE)\to A \to \mathcal{PE}\to 0$ where $A$ is the pull-back of the diagram 
 \[
   \xymatrix{
  A \ar[d] \ar[r] & \mathcal{PE}\ar[d]^a \\
  \Omega^1_X\otimes\EE \ar[r] & F^1
 }
 \]
 In particular the composition 
 \[
 \hom_{\OO_X}(\mathcal{PE},F^1 )\xrightarrow{\delta} \ext^1_{\OO_X}(\mathcal{PE},\mathrm{Sym}^2(\EE))\to \ext^1_{\OO_X}(\Omega^1_X\otimes\EE,\mathrm{Sym}^2(\EE))
 \]
 evaluated at the element $[\nabla] \in \hom_{\OO_X}(\mathcal{PE},F^1 )$ returns the isomorphism class of the extension $0\to \mathrm{Sym}^2(\EE) \to B \to \Omega^1_X\otimes\EE$ where $B$ is the pull-back of the diagram
 \[
   \xymatrix{
  B \ar[d] \ar[r] & \Omega^1_X\otimes\EE\ar[d]^{[\nabla]\circ i} \\
  \Omega^1_X\otimes\EE \ar[r] & F^1
 }
 \]
 where $i:\Omega^1_X\otimes\EE\to\mathcal{PE}$ is the canonical immersion.
 
 Now to compute $[\nabla]\circ i:\Omega^1_X\otimes\EE\to F^1$ recall that $[\nabla]$ is defined by being the unique $\OO_X$-linear morphism making the following diagram commute,
 \[
   \xymatrix{
  \EE \ar[d]^{d^{(1)}} \ar[r]^d & F^1 \\
  \mathcal{PE} \ar[ru]_{[\nabla]} & 
 }
 \]
 where $d^{(1)}:\EE\to \mathcal{PE}$ is the universal differential operator of order $1$. Then, as follows from the formulas of \cite[p.~193]{atiyah} explicitly describing the $\OO_X$-module structure of $\mathcal{PE}$, given local sections $f$ of $\OO_X$ and $\varpi$ of $\EE$ we have 
 \[
 [\nabla](df\otimes \varpi) = d(f\varpi)-fd(\varpi).
 \]
 So, $[\nabla]\circ i$ is just the exterior product of forms, hence the sequence $0\to \mathrm{Sym}^2(\EE) \to B \to \Omega^1_X\otimes\EE\to 0$ splits, then the class of $\delta([\nabla])$ in  $\ext^1_{\OO_X}(\Omega^1_X\otimes\EE,\mathrm{Sym}^2(\EE))$ is zero. Therefore $\delta([\nabla])$ is in the image of $\ext^1_{\OO_X}(\EE,\mathrm{Sym}^2(\EE))\to \ext^1_{\OO_X}(\mathcal{PE},\mathrm{Sym}^2(\EE))$. Hence if  $\ext^1_{\OO_X}(\EE,\mathrm{Sym}^2(\EE))=(0)$ then $\delta([\nabla])=0$. 
 So, if $\ext^1_{\OO_X}(\EE,\mathrm{Sym}^2(\EE))=(0)$, then there is a morphism $\mathcal{PE}\to \Omega^1_X\otimes\EE$ lifting $[\nabla]$.

 What we need to prove to conclude is that among the morphisms $\mathcal{PE}\to \Omega^1_X\otimes\EE$, there is one $\nabla:\mathcal{PE}\to \Omega^1_X\otimes\EE$ such that $\nabla|_{\Omega^1_X\otimes \EE}$ is the identity.
 To do this we consider the short exact sequence $0\to \Omega^1_X\otimes\EE\to \mathcal{PE}\to \EE\to 0$ and the exact sequence of $\hom$ groups
 \begin{align*}
  \hom_{\OO_X}(\EE,\Omega_X\otimes\EE)\to\hom_{\OO_X}(\mathcal{PE},\Omega_X\otimes\EE)&\to\hom_{\OO_X}(\Omega_X\otimes\EE,\Omega_X\otimes\EE)\xrightarrow{\delta} \\
  &\xrightarrow{\delta}\ext^1(\EE,\Omega_X\otimes\EE)\to\cdots
 \end{align*}
 The identity is an element $\mathrm{id}\in \hom_{\OO_X}(\Omega_X\otimes\EE,\Omega_X\otimes\EE)$ and we want to show that it is the restriction of some morphism $\nabla:\mathcal{PE}\to  \Omega_X\otimes\EE$, which is equivalent to the condition $\delta(\mathrm{id})=0$. 
 We already know that there is a morphism $\tilde{\nabla}:\mathcal{PE}\to \Omega^1_X\otimes\EE$ lifting $[\nabla]$, so the restriction of $\tilde{\nabla}$ to $\Omega^1_X\otimes\EE$, which we also denote $\tilde{\nabla}$, makes the following diagram commute.
 \[
   \xymatrix{
  0 \ar[r] &  \mathrm{Sym}^2(\EE) \ar[d] \ar[r] & \Omega^1_X\otimes\EE \ar[d]^{\tilde{\nabla}} \ar[r] & F^1 \ar[r] \ar@{=}[d] & 0 \\
  0 \ar[r] &  \mathrm{Sym}^2(\EE)  \ar[r] & \Omega^1_X\otimes\EE  \ar[r] & F^1 \ar[r] & 0
 }
 \]
 Being a restriction we have $\delta(\tilde{\nabla})=0$.
 If we can prove that $\delta(\tilde{\nabla}-\mathrm{id})=0$ then $\delta(\mathrm{id})=0$ and we are set.
 The image of $\tilde{\nabla}-\mathrm{id}$ is in $\mathrm{Sym}^2(\EE)$ so the element $\delta(\tilde{\nabla}-\mathrm{id})\in \ext^1(\EE,\Omega_X\otimes\EE)$ is the class of the extension in the last row of the diagram:
 \[
   \xymatrix{
  0 \ar[r] &  \Omega^1_X\otimes\EE \ar[d]^{\tilde{\nabla}-\mathrm{id}} \ar[r] & \mathcal{PE} \ar[d] \ar[r] & \EE \ar[r] \ar[d] & 0 \\
  0 \ar[r] &  \mathrm{Sym}^2(\EE) \ar[d] \ar[r] & I \ar[d] \ar[r] & \EE \ar[d] \ar[r] & 0 \\
  0 \ar[r] &  \Omega^1_X\otimes\EE \ar[r] & II \ar[r] & \EE \ar[r] & 0
 }
 \]
 Where $I$ and $II$ are push-forwards. If  $\ext^1_{\OO_X}(\EE,\mathrm{Sym}^2(\EE))=(0)$ then $[I]=0\in \ext^1(\EE,\mathrm{Sym}^2(\EE))$ and so $\delta(\tilde{\nabla}-\mathrm{id})=[II]=0\in \ext^1(\EE,\Omega^1_X\otimes\EE)$. 
 So the condition  $\ext^1_{\OO_X}(\EE,\mathrm{Sym}^2(\EE))=(0)$ implies that there is a connection on $\EE$.
\end{proof}

\begin{remark} In the case where $\EE=\LL$ is a line bundle (\emph{i.e.} the codimension $1$ case) we have $\mathrm{Sym}^2(\LL)=\LL^{\otimes 2}$ so $\ext^1(\LL, \mathrm{Sym}^2(\LL))=\ext^1(\LL, \LL^{\otimes 2})=\ext^1(\OO_X, \LL)=H^1(X,\LL)$. So in the codimension $1$ case we recover Proposition \ref{propc1L=0}. 
\end{remark}

\begin{corollary}
Let $X$ be a smooth projective variety over $\CC$ such that every line bundle $\LL$ verifies $H^1(X,\LL)=0$ and such that $Pic(X)$ is torsion-free (\emph{e.g.}: $X$ smooth complete intersection). And let $\LL\xrightarrow{\omega}\Omega^q_X$ be a foliation such that $\EE\cong\bigoplus_i \LL_i$ for some line bundles $\LL_i$. Then $\LL$ has persistent singularities.
\end{corollary}
\begin{proof}
 As $\EE$ is a direct sum of line bundles the group $\ext^1_X(\EE,\mathrm{Sym}^2(\EE))$ decomposes as
 \[
  \ext^1_X(\EE,\mathrm{Sym}^2(\EE))=\bigoplus_{i}\bigoplus_{j\le k}
  H^1(X,\LL_i^{-1}\otimes \LL_j\otimes\LL_k)=0.
 \]
 So, as $\omega$ is decomposable, if it does not posses persistent singularities then $\II=\OO_X$ so by Theorem \ref{teoConnectionE} there is a connection on $\EE$. This implies that the Chern classes of the line bundle $\LL_1,\dots,\LL_q$ are all zero. Then $\LL_i\cong\OO_X$ for $i=1,\dots,q$, giving global sections of $\Omega^1_X$, contradicting the fact that $H^0(X,\Omega^1_X)=H^1(X,\OO_X)=0$.
\end{proof}

\begin{remark}
 The main theorem of this section gives criteria that, provided strong hypotheses of cohomological nature on the sheaf $\EE$, the absence of persistent singularities implies $\EE$ is a locally free sheaf. The question of when a foliation can be defined by a locally free sheaf $\EE$ is a particularly interesting one. 
 It follows in a similar fashion as in example \ref{exampleCodimension2inA3} that a general codimension $2$ foliation in a $3$ dimensional space (or more generally a dimension $1$ foliation disregarding the dimension of the ambient space) cannot have $\EE$ locally free around an isolated singularity. So restricting a foliation with both persistent and isolated non-persistent singularities (as foliations defined by logarithmic $2$-forms generally are) to the open space $U\subseteq X$ given by the complement of the persitent singularities we would have a foliation with $\II=\OO_X$ and $\EE$ not locally free.
 We do not know, however, of an example where $\II=\OO_X$ and $\EE$ is \emph{not} locally free in the projective case. So the following question arises.
 
 \
 
 \textbf{Question:} Let $X$ be a projective variety. Does $\mathcal{I}=\OO_X$ implies $\EE$ is locally free? Under what hypotheses on $X$ this is true?
 \end{remark}

\

\

\begin{tabular}{l l}
C\'esar Massri$^*$ \hspace{3cm}\null&\textsf{cmassri@dm.uba.ar}\\
Ariel Molinuevo$^\dag$  &\textsf{arielmolinuevo@gmail.com}\\
Federico Quallbrunn$^\ddag$  &\textsf{fquallb@dm.uba.ar}\\
\end{tabular}

\

\begin{tabular}{l l}
&\\
$^*${Departamento de Matem\'atica} & \\
{Pabell\'on I} &\\
{Ciudad Universitaria} &\\
{CP C1428EGA} &\\
{Buenos Aires} &\\
{Argentina} &\\

&\\
$^\dag$Instituto de Matem\'atica \\
Universidade Federal do Rio de Janeiro \\
Caixa Postal 68530 \\
CEP. 21945-970 Rio de Janeiro - RJ \\
BRASIL \\
&\\
\end{tabular}

\

\begin{tabular}{l l}
$^\ddag$ Departamento de Matem\'atica & \\
Universidad CAECE & \\
Av. de Mayo 866 & \\
CP C1084AAQ & \\
{Buenos Aires} &\\
{Argentina} &\\
\end{tabular}

\end{document}